\documentclass[20pt]{article}

\input epsf
\usepackage{amssymb}
\usepackage{amsfonts}
\usepackage{amsmath}
\usepackage{mathrsfs}
\usepackage{amsthm}

\numberwithin{equation}{section}
\newtheorem{theorem}{Theorem}[section]
\newtheorem{lemma}{Lemma}[section]

\newtheorem{corollary}{Corollary}[section]

\newtheorem{remark}{Remark}[section]

\begin{document}

 \baselineskip 20pt

\textwidth=146truemm \textheight=207truemm

\begin{center}
{\bf \Large The asymptotic properties of Eulerian numbers and
refined Eulerian numbers: A Spline perspective \footnote{This work
was partially Supported by The National Natural Science Foundation
of China (NO.60373093, NO.60533060, NO.10801024, U0935004 and
NO.10726068.) and the Innovation Fundation of the Key Laboratory of
High Temperature Gasdynamics of Chinese Academy of Sciences. This
research project was carried while the second author visited
Michigan State University in 2008-09 under the support of State
Scholarship Fund by China Scholarship Council in 2008. }}
\end{center}
\begin{center}
{Renhong Wang \footnote{Email: renhong@dlut.edu.cn}    Yan Xu
\footnote{Email: yanxudlut@yahoo.cn}

 {\small Institute of
Mathematical Sciences, Dalian University of Technology,\newline
 Dalian, 116024, China} }

\end{center}

\begin{abstract}In this paper, the asymptotic formulas for Eulerian numbers, refined
Eulerian numbers and the coefficients of descent polynomials are
obtained directly from the spline interpretations of these numbers.
Having related these numbers directly to B-splines \cite{xu}, we
   can take advantage of many powerful spline techniques to derive various
   results of these numbers.
 The asymptotic formulas for the Eulerian numbers $A_{d, k}$ agree with the
 previously known results which were given by L. Carlitz et al.(1972)\cite{Carliz} and S.Tanny (1973) \cite{Tanny},
  but the convergence order is much better. We also give the asymptotic representations
  of refined Eulerian numbers which is in terms of the Hermite polynomials.

\end{abstract}


\vspace{0.5cm}
\section{ Introduction}
Eulerian number, denoted here by $A_{d,k}$ is the number of
permutations in the symmetric group $S_d$ that have exactly $k-1$
descents. They play important roles in enumerative combinatorics,
geometry \cite{l}, statistical applications \cite{Carliz,Tanny}, and spline theory\cite{scheonberg,xu}.
The asymptotic properties of Eulerian numbers
was examined by S. Tanny \cite{Tanny} and L. Carlitz et. al.\cite{Carliz} with the help of central Limit theorem of probability theory.
They showed that the Eulerian numbers approximate to the Gaussian function with the convergence order $O(d^{-\frac{3}{4}})$. Nevertheless, in this paper, we obtain a much better convergence order $O(d^{-\frac{3}{2}})$ with the help of the spline theory.

The refined Eulerian number $\mathbf{A}_{d,k,j}$ is the number of
permutations in the symmetric group $S_d$ with $k$ descents and
ending with the element $j$. Brenti and Welker \cite{Brenti} gave a precise description of the h-vector
of the barycentric subdivision of a simplicial complex in terms of the
h-vector of the original complex and the refined Eulerian numbers. Later, Martina Kubitzke and Eran Nevo \cite{Martina}
derived new inequalities on $\mathbf{A}_{d,k,j}$ by the algebraic method. A result due to Ehrenborg,
Readdy and  Steingr\'{\i}msson \cite{E} gave a geometry interpolation of $\mathbf{A}_{d,k,j}$ as the mixed volumes of two adjacent slices from the unit cube. In \cite{xu} Renhong,W., Yan, X. and Zhiqiang, X. gave a spline interpolation of the $\mathbf{A}_{d,k,j}$ by the geometry interpolation and the spline theory. Further more, they gave the explicit and recurrence representation of $\mathbf{A}_{d,k,j}$. In \cite{mark}, the experimental evidence suggested that the refined Eulerian number $\mathbf{A}_{d,k,j}$ is not exactly normal approximation as $d$ grows large. In this paper, we give the asymptotic representations of refined Eulerian numbers in terms of the Hermite polynomials.

Recently, Steingr\'{\i}msson \cite{Stein} generalized the definition of descents and excedances to the elements (called indexed permutations, colored permutations or $r$-signed permutations see\cite{E,Stein,Bagno2,Bagno1} ) of groups $S^n_d:= Z_n \wr S_d$, where $\wr$ is wreath product with respect to the usual action of $S_d$ by permutations of $[d]$.
The number of indexed permutations in $S^n_d$
with $k$ descents is denoted by $D(d,n,k)$.
 In \cite{Stein}, using the work of Brenti \cite{Benti2}, Steingr\'{\i}msson \cite{Stein} showed that the numbers $D(d,n,k)$ are \textit{unimodal}.
Moreover, In\cite{xu}, using the spline theory, the authors gave the explicit expression of $D(d,n,k)$ and showed the numbers $D(d,n,k)$ are log-concave. In this paper, we show that $D(d,n,k)$ is approximately normal in terms of the spline interpretation of $D(d,n,k)$.

The paper is organized as follows. In Section 2, after recalling some necessary definitions and
notations, we show the connection between the \textit{B-splines} and these combinatorial numbers.
In Section 3, we give the main results of this paper. The  results show the development of
various asymptotic properties of Eulerian numbers, refined Eulerian
numbers and the coefficients of descent polynomials, which are
obtained directly from the spline interpretations of these numbers. We prove the main results in Section 4.
\section{definitions and notations}

Eulerian number, denoted here by $A_{d,k}$ is the number of
permutations in the symmetric group $S_d$ that have exactly $k-1$
descents. An recurrence formula for $A_{d,k}$ is
\begin{equation}
A_{d,k+1} =(k+1)A_{d-1,k+1}+(d-k)A_{d-1,k}
\end{equation}
with the boundary conditions
$$
A_{0,0}=1, A_{d,0}=0, d>0.
$$
An explicit formula for $A_{d,k}$ is
\begin{equation}\label{eq:EN}
A_{d,k}=\sum_{i=0}^k\binom{d+1}{i}(-1)^i(k-i)^{d}
\end{equation}
which can be easily verified using the above recurrence.

Refined Eulerian number $\mathbf{A}_{d,k,j}$ is the number of
permutations in the symmetric group $S_d$ with $k$ descents and
ending with the element $j$.
The recurrence formulas for $\mathbf{A}_{d,k,j}$ are \cite{xu}
\begin{equation}\label{eq:recrefine}
\begin{split}
&\mathbf{A}_{d+1,k,d-j+1}=(k+1)\mathbf{A}_{d,k,d-j}
+(d-k)\mathbf{A}_{d,k-1,d-j},\\
&\mathbf{A}_{d+1,k,d-j+1}=k\mathbf{A}_{d,k,d-j+1}
+(d-k+1)\mathbf{A}_{d,k-1,d-j+1}.\\
\end{split}
\end{equation}
An explicit formula for $\mathbf{A}_{d,k,j}$ is \cite{xu}
\begin{eqnarray*}
\mathbf{A}_{d,k,j}&=&\sum_{i=0}^{k}\binom{d}{i}(-1)^i(k-i)^{d-j}(k-i+1)^{j-1}.\\
 \end{eqnarray*}

Descent polynomials, denoted by $D^n_d(t)$, are defined as
$$
D^n_d(t)=\sum_{k=0}^d D(d,n,k)t^k,
$$
where $D(d,n,k)$  is the number of indexed permutations in $S^n_d$
with $k$ descents. The indexed permutation of length $d$ and with
indices in $\{0,1,\ldots,n-1\}$ is an ordinary permutation in the
symmetric group $S_d$ where each letter has been assigned an integer
between $0$ and $n-1$. Indexed permutations, or $r$-signed
permutations, are a generalization of permutations (see
\cite{E,Stein}). We will follow the notation in \cite{Stein}. The
set of all such indexed permutations is denoted by $S^n_d$.
The numbers $D(d,n,k)$ satisfy a simple three-term
recurrence\cite{Stein},
\begin{equation}
 D(d,n,k)=(nk+1)D(d-1,n,k)+(n(d-k)+(n-1))D(d-1,n,k-1).
\end{equation}
And also has an explicit formula \cite{xu}:
\begin{eqnarray*}
D(d,n,k)=\sum_{i=0}^{k}\binom{d+1}{i}(-1)^i(n(k-i)+1)^d.\\
 \end{eqnarray*}
We turn to the definitions of B-splines and Hermite polynomials,
then show the relations
 among Eulerian numbers, the refined Eulerian numbers, the descent polynomials and B-splines.

Let $p\in \mathbb{R}^{\geq 1}$ and $L^p(\mathbb{R})$ as usual denote
the set
\begin{eqnarray*}
 L^p(\mathbb{R}) = \{ f : \mathbb{R} \rightarrow \mathbb{C}|
f & measurable,& \int_{-\infty}^{+\infty}\left|f(t)\right|^p dt <
\infty \}.
\end{eqnarray*}
For $p = 2$ and $f, g \in L^2(\mathbb{R})$ define the inner product
\begin{equation*}
\left\langle f,g\right\rangle
:=\int_{-\infty}^{+\infty}f(t)\overline{g(t)} dt
\end{equation*}
and the norm
$$ \left\| f\right\|:=\sqrt{\langle f,f \rangle}$$
making $L^2(\mathbb{R})$ to a Hilbert space.

For $f \in  L^1(\mathbb{R})$ define the Fourier transform $f^\wedge$
and the inverse Fourier transform $f^\vee$ as
\begin{eqnarray*}
f^\wedge(\omega) \,&:=&\, \int_{-\infty}^{+\infty} f (t) e^{-i\omega
t}dt,\\
 f ^\vee(\omega) \,&:=&\, \frac{1} {2\pi}
\int_{-\infty}^{+\infty} f (t) e^{-i\omega t} dt.\\
\end{eqnarray*}
This definition can be extended to functions $f\in L^2(\mathbb{R})$,
see for example \cite{Stein W}.

{\it  B-splines with order $d$}, which is denoted as $B_d(\cdot)$,
is defined by the induction as
$$
B_1(x)\,\,=\,\,\begin{cases}
1 & \text{if} \>\> x\in[0,1),\\
0 & \text {otherwise},
\end{cases}
$$
and for $d\geq 1$
$$
B_d\,\,=\,\, B_1*B_{d-1},
$$
where $*$ denotes the operation of convolution which is defined:
$$
( f\ast g)(t) := \int_{-\infty}^{+\infty} f (t -y)g(y)dy
$$
for $f $ and $g$ in $L^2(\mathbb{R})$.

Then $B_d$ has the compact support $[0, d]$ and is in
$C^{d-1}(\mathbb{R})$. A well known explicit formula for $B_d(\cdot)$ is
\begin{equation}\label{eq:bspline}
B_d(x)=\frac{1}{(d-1)!}\sum_{i=0}^{d}(-1)^i\binom{d}{i}(x-i)_+^{d-1}
\end{equation}
where the one-sided power function is defined by:
\begin{equation*}
x_+^d=\left\{\begin{array}{lll} x^d & &
 x \geq 0, \\
0& & otherwise.
\end{array} \right.
\end{equation*}
And the recurrence relation is
\begin{equation}\label{eq:recbspline}
B_d(x)=\frac{x}{d-1}B_{d-1}(x)+\frac{d-x}{d-1}B_{d-1}(x-1).
\end{equation}
For all $B_d$, $d\geq1$, it holds $B_d\in L^1(\mathbb{R})\bigcap
L^2(\mathbb{R}) $  and
\begin{equation}
B^\wedge_d(\omega)=sinc^{d}\left(\frac{\omega}{2}\right)
\end{equation}
where $sinc(t)$ denotes:
\begin{equation}
sinc(t):=\left\{\begin{array}{lll}
\frac{\sin{t}}{t} & & t\neq 0, \\
1& & t=0.
\end{array} \right.
\end{equation}
B-splines play important roles in approximation, computer-aid design(CAD), signal
processing and discrete geometry. They have been well developed during the past few decades. For extensive monographs see
\cite{deboorBspline,zhiqiang}.

{\it  Hermite polynomials $H_n(x)$ with degree $d$}, are set of
orthogonal polynomials over the domain $(-\infty,+\infty)$  with
weighting function $e^{-x^2}$. The Hermite polynomials are defined
by
\begin{equation}\label{H}
H_n(x)=(-1)^n e^{\frac{x^2}{2}}\frac{d^n}{dx^n}e^{\frac{-x^2}{2}}
\end{equation}
and satisfy
\begin{eqnarray*}
H_n(x)\,&=&\,\sum_{k=0}^{[\frac{n}{2}]}\frac{(-1)^k
n!(2x)^{n-2k}}{k!(n-2k)!}.\\
\end{eqnarray*}
Being the limiting case of several families of classical orthogonal
polynomials, they are of fundamental importance in asymptotic
analysis.

 In
this paper, we shall be concerned primarily with the development of
various asymptotic properties of Eulerian numbers, refined Eulerian
numbers and the coefficients of descent polynomials, which are
obtained directly from the spline interpretations of these numbers
\cite{xu}. The theory of B-splines is a well developed area of
applied numerical analysis and interpolation theory. Having related
these numbers directly to B-splines \cite{xu}, we can take advantage
of many powerful spline techniques to derive various results of
these numbers. It is precisely this fact which motivates the need
for more purely spline interpretations of these numbers to which we
now turn. To state conveniently,  we use $[\lambda^j]f(\lambda)$ to
denote the coefficient of $\lambda^j$ in $f(\lambda)$ for any given
power series $f(\lambda)$.

\begin{lemma}\label{le:1}\cite{xu}
$$
  A_{d,k} \,=\, d!\cdot B_{d+1}(k);\leqno{\rm (i)}
$$
$$
 D(d,n,k)\,=\, d!\cdot n^d\cdot B_{d+1}\left(k+\frac{1}{n}\right);\leqno{\rm (ii)}
$$
$$
   \mathbf{A}_{d+1,k,d-j+1}\, =\, d!\cdot [
\lambda^j]\left(
(\lambda+1)^dB_{d+1}\left(k+\frac{1}{\lambda+1}\right)\right)/{\binom{d}{j}},\,\,
\lambda\geq 0.\leqno{\rm (iii)}
$$
\end{lemma}
\section{Main results}
After the necessary definitions have been provided, we can come to
the main results of this article.

\begin{theorem}\label{limtB} Let be $k\in
\mathbb{N}$ then for $d> k+2$ the sequence of the $k$-th derivatives
$B^{(k)}_{d}$ of the $B$-spline converges to the $k$-th derivative
of the Gaussian function,
\begin{equation}
\left(\frac{d}{12}\right)^\frac{k+1}{2}B^{(k)}_{d}\left(\sqrt{\frac{d}{12}}x+\frac{d}{2}\right)
=\frac{1}{\sqrt{2\pi}} D^k\exp\left({-\frac{x^2}{2}}\right)
+O\left(\frac{1}{d}\right),
\end{equation}
and
\begin{equation}
\lim_{d\rightarrow\infty}\left\{\left(\frac{d}{12}\right)^\frac{k+1}{2}B^{(k)}_{d}
\left(\sqrt{\frac{d}{12}}x+\frac{d}{2}\right)\right\}
=\frac{1}{\sqrt{2\pi}} D^k\exp\left({-\frac{x^2}{2}}\right) ,
\end{equation}
where the limit may be taken point-wise or in $L^p(\mathbb{R}), p\in
[2,\infty)$.
\end{theorem}
Sommerfeld in 1904 \cite{A.Sommerfeld} show that the Gaussian
function can be approximated by B splines point-wise. In 1992, Unser
and colleagues\cite{Unser} proved that the sequence of normalized and scaled
B-splines $B_d$ tends to the Gaussian function as the order $d$
increases in $L^p(\mathbb{R})$. A result due to Ralph Brinks \cite{RALPH
BRINKS} generalize Unser's result to the derivatives of the
B-splines. In this paper, we reprove the theorem and give the order
of the convergence.

\begin{theorem}\label{th1}
For $x_d=\sqrt{\frac{d+1}{12}}x+\frac{d+1}{2}$, we have
\begin{equation}\label{eq:A}
\frac{1}{d!}A_{d,[x_d]}
=\sqrt{\frac{6}{\pi(d+1)}}\exp\left({-\frac{x^2}{2}}\right)+O\left(d^{-\frac{3}{2}}\right)
\end{equation}
\end{theorem}
L. Carlitz, D. C. Kurtz, R. Scoville and O. P. Stackelberg  in
\cite{Carliz}
 showed (\ref{eq:A}) with the help of the central limit theorem of
probability theory that the expression on the right side of
(\ref{eq:A}) has the order $O(d^{\frac{-3}{4}})$.
 Nevertheless, using the spline interpretation of Eulerian numbers, we obtain
  the same asymptotic forms of $A_{n, k}$ with a much better convergence order $O(d^{-\frac{3}{2}})$.
\begin{theorem}\label{corollary1}
For $x_d=\sqrt{\frac{d+1}{12}}x+\frac{d+1}{2}$, we have
\begin{equation}\label{eq:7}
\frac{1}{d!\cdot n^d}D(d,n,[x_d])=
\sqrt{\frac{6}{\pi(d+1)}}\exp\left({-\frac{\left(x+\frac{1}{n}\right)^2}{2}}\right)+O\left(d^{-\frac{3}{2}}\right)
\end{equation}
\end{theorem}
Using the spline theory, we can also get the asymptotic
representations
  of refined Eulerian numbers in terms of the Hermite polynomials.
\begin{theorem}\label{ThREN}
Let $x_d=\sqrt{\frac{d+1}{12}}\left(x-1\right)+\frac{d+1}{2}$ then
\begin{eqnarray*}
\mathbf{A}_{d+1,[x_d],d-j+1}
\, \,&=&\,d!\sqrt{\frac{6}{\pi
(d+1)}}\exp{\left(-\frac{x^2}{2}\right)}\sum_{i=0}^j
\frac{1}{\binom{d-j+i}{i}}\left(\frac{d+1}{12}\right)^{-\frac{i}{2}}H_i(x)+O\left(d^{\frac{-3}{2}}\right),\\
\end{eqnarray*}
where $H_n(x)$ are the Hermite polynomials are defined by
\begin{equation}
H_n(x)=(-1)^n e^{\frac{x^2}{2}}\frac{d^n}{dx^n}e^{\frac{-x^2}{2}}.
\end{equation}
\end{theorem}

\section{Proofs of the main results}
The proof of Theorem \ref{limtB} is based on the following lemmata
given in \cite{RALPH BRINKS} which shows an upper bound for
standardized Sinc functions.
\begin{lemma}\label{lemma1}\cite{RALPH BRINKS}
For $k, d\in\mathbb{N}$ and $d\geq k+2 $, there is a constant
$c_k\in \mathbb{R}^+$ such that for
\begin{equation*}
G_k(x):=\chi_{{\mathbb{R}\setminus[-1,1]}}(x)\frac{c_k}{\pi^2x^2}+\pi^k|x|^k\exp(-x^2)
\end{equation*}
where $\chi_A(x)$ is characteristic function on set $A$, it holds
\begin{equation*}
G_k\in L^p(\mathbb{R}),\forall p\in [1,\infty),
\end{equation*}
and
\begin{equation*}
\pi^k|x|^k\cdot\left|sinc\left(\frac{\pi
x}{\sqrt{d}}\right)\right|^d\leq G_k(x).
\end{equation*}
\begin{proof}
We start with the inequality
\begin{equation*}
\forall x\in [0,1], sinc(x)\leq 1-x^2,
\end{equation*}
where the right-hand side term is  the parabola that goes through the two extreme
values of $sinc(x)$ within this interval. We will first show that
\begin{equation}\label{eqlep}
\forall x\in [0,\sqrt{d}],
sinc^n\left(x/\sqrt{d}\right)\leq\left(1-\frac{x^2}{d}\right)^d\leq
\exp\left(-x^2\right).
\end{equation}
For this purpose, we define the positive function
\begin{equation*}
p(x)=\exp\left(x^2\right)\left(1-\frac{x^2}{d}\right)^d.
\end{equation*}
The derivative of $p(x)$ is given by
\begin{equation*}
\frac{\partial p(x)}{\partial
x}=\frac{-2\exp(x^2)x^3\left(1-x^2/d\right)^d}{d\left(1-x^2/d\right)}
\end{equation*}
and is always negative for $x\in [0,\sqrt{d}]$. Therefore, the
maximum of $p(x)$ within the interval occurs at $x=0$
\begin{equation*}
\sup_{x\in (0,\sqrt{n})}p(x)=p(0)=1,
\end{equation*}
which proves (\ref{eqlep}).

Second, for $x\geq \sqrt{d}$, set
\begin{equation*}
c_k:= \max_{n\geq
k+2}\left(\frac{d^{\frac{k+2}{2}}}{\pi^{d-k-2}}\right).
\end{equation*}
We note that $\pi^k|x|^k\cdot \left|sinc(\pi x/\sqrt{d})\right|^d$
is also bounded by
\begin{equation}\label{eqleq2}
\pi^k|x|^k\cdot \left|sinc\left(\pi
x/\sqrt{d}\right)\right|^d\leq\frac{c_k}{x^{2}\pi^{2}}.
\end{equation}
The existence of $c_k$ follows from the convergence of the sequence
\begin{equation*}
\left(\frac{d^{\frac{k+2}{2}}}{\pi^{d-k+2}}\right)_{d\geq {k+2}}.
\end{equation*}
Then one follows
\begin{equation*}
\left(\frac{\sqrt{d}}{x\pi}\right)^d \leq
\left(\frac{\sqrt{d}}{x}\right)^{k+2}\cdot
 \frac{1}{\pi^d} \leq \frac{c_k}{x^{k+2}\pi^{k+2}}.
\end{equation*}
 We get a bound $\frac{c_k}{x^{2}\pi^{2}}$ that is independent of $d$ by noticing that
 \begin{equation*}
\left|sinc^d\left(\frac{\pi x}{\sqrt{d}}\right)\right| \leq \left(
\frac{\sqrt{d}}{\pi x}\right)^d.
 \end{equation*}

 Finally, we define the function $G_k(x)$, which is independent of $d$,
 by suitably combining the right-hand sides of (\ref{eqlep}) and (\ref{eqleq2}).
 The function $\pi^k|x|^k\cdot\left|sinc\left(\frac{\pi x}{\sqrt{n}}\right)\right|^n $ is uniformly
 bounded from above by an $L_p(-\infty,+\infty)$ function $G_k(x)$ where $p\in[1,+\infty)$.
\end{proof}
\end{lemma}

We have all the ingredients to prove our results now.
\begin{proof}[Proof of Theorem \ref{limtB}]
Set
\begin{equation}\label{eq:2.6}
L_n(x):=d\ln
\left(sinc\left(\frac{x}{2}\sqrt{\frac{12}{d}}\right)\right).
\end{equation}
Due to the symmetry, we may assume $x \geq 0$. By Taylor's theorem,
for any $x \in[0,1]$ and $d\in \mathbb{N}$, it holds
\begin{equation*}
sinc\left(\frac{x}{2}\sqrt{\frac{12}{d}}\right)=1-\frac{x^2}{2}\frac{1}{d}+
O\left(\frac{1}{d^2}\right)
\end{equation*}
and
\begin{equation*}
\ln(1+x)=x+O(x^2).
\end{equation*}
Then for any $x\in [0,1]$ and $d\in \mathbb{N}$, it holds
\begin{equation}\label{eq:2.7}
L_n(x)=d\ln\left(1-\frac{x^2}{2}\frac{1}{d}+
O\left(\frac{1}{d^2}\right)\right) =-\frac{x^2}{2}+
O\left(\frac{1}{d}\right).
\end{equation}
Combing (\ref{eq:2.6}) and (\ref{eq:2.7}), we have
\begin{equation}\label{eq:2.8}
sinc^d\left(\frac{x}{2}\sqrt{\frac{12}{d}}\right)=
\exp\left({-\frac{x^2}{2}}\right)\left(1 +
O\left(\frac{1}{d}\right)\right).
\end{equation}
Furthermore, it holds
$B^\wedge_d(\omega)=sinc^{d}\left(\frac{\omega}{2}\right)$, and $B_d
\in C^{d-1}(\mathbb{R}).$ The later yields for $k\leq d-1$:
\begin{equation*}
\left[B^{(k)}_{d}\right]^\wedge(\omega)
=i^k\omega^k sinc^{d}\left(\frac{\omega}{2}\right).\\
\end{equation*}
Consequently, one obtains
\begin{eqnarray*}
\left(\frac{d}{12}\right)^\frac{k+1}{2}\left[B^{(k)}_{d}\left(\sqrt{\frac{d}{12}}x+\frac{d}{2}\right)\right]^\wedge(\omega)
\,&=&\,i^k\omega^k sinc^{d}\left(\frac{\omega}{2}\sqrt{\frac{12}{d}}\right)\\
\,&=&\,i^k\omega^k \exp\left({-\frac{\omega^2}{2}}\right)\left(1+O\left(\frac{1}{d}\right)\right)\\
\,&=&\,\left[D^k\left\{\frac{1}{\sqrt{2\pi}}\exp{\left(-\frac{x^2}{2}\right)}\cdot
\left(1+O\left(\frac{1}{d}\right)\right)\right\}\right]^\wedge(\omega)
\\
\end{eqnarray*}
\begin{eqnarray*}
\left(\frac{d}{12}\right)^\frac{k+1}{2}B^{(k)}_{d}\left(\sqrt{\frac{d}{12}}x+\frac{d}{2}\right)
\,&=&\,D^k\left(\frac{1}{\sqrt{2\pi}}\exp\left({-\frac{x^2}{2}}\right)\right)
+D^k\left(\frac{1}{\sqrt{2\pi}}\exp\left({-\frac{x^2}{2}}\right)\right)O\left(\frac{1}{d}\right)\\
\,&=&\,D^k\left(\frac{1}{\sqrt{2\pi}}\exp\left({-\frac{x^2}{2}}\right)\right)
+O\left(\frac{1}{d}\right)\\
\end{eqnarray*}
for $d\rightarrow\infty$
\begin{equation}\label{eq:limt1}
lim_{d\rightarrow\infty}
\left\{\left(\frac{d}{12}\right)^\frac{k+1}{2}B^{(k)}_{d}\left(\sqrt{\frac{d}{12}}x+\frac{d}{2}\right)\right\}
=\frac{1}{\sqrt{2\pi}} D^k\exp{\left(-\frac{x^2}{2}\right)}
\end{equation}
where the limit in equation (\ref{eq:limt1}) follows from Lemma
\ref{lemma1} and is meant to be taken point-wise. Clearly,
$B^{(k)}_{d}(\omega)$ is bounded by the $L_p$ function $G_k(x)$
defined in Lemma \ref{lemma1} and this bound is independent of $d$.
The use of Lebesgue's dominated convergence Theorem provides the
$L_p$ convergence for $P\in[1,+\infty)$. The $L_q$ convergence with
$q\in[2, +\infty]$ in the time domain follows as a consequence of
Titchmarsh inequality, which states that for $1\leq  p \leq 2$ and
$p^{-1} + q^{-1} = 1$, the Fourier transform is a bounded linear
operator from $L_p( -\infty, + \infty)$ into $L_q(-\infty,
+\infty)$.
\end{proof}

When $k=0$, the Theorem \ref{limtB} turns to
\begin{corollary}\label{corollary2} For $d\in \mathbb{N}$, the
B-spline converges to the Gaussian function,
\begin{equation}
\sqrt{\frac{d}{12}}B_{d}\left(\sqrt{\frac{d}{12}}x+\frac{d}{2}\right)=\frac{1}{\sqrt{2\pi}}
\exp{\left(-\frac{x^2}{2}\right)}+O\left(\frac{1}{d}\right),
\end{equation}
and
\begin{equation}
\lim_{d\rightarrow\infty} \left\{
\sqrt{\frac{d}{12}}B_{d}\left(\sqrt{\frac{d}{12}}x+\frac{d}{2}\right)\right\}=\frac{1}{\sqrt{2\pi}}
\exp{\left(-\frac{x^2}{2}\right)},
\end{equation}
where the limit may be taken point-wise or in $L^p(\mathbb{R}),
p\in[2,\infty).$
\end{corollary}

\begin{proof} [Proof of Theorem \ref{ThREN}] We firstly prove Theorem
\ref{ThREN}.

Set
$$
p(\lambda)=\sum_{j=0}^d(_j^d)\mathbf{A}_{d+1,k,d-j+1}\lambda^j
$$ for
$\lambda >-1$. Then $p(\lambda)$ is the polynomial of $\lambda$ with
$d$ degree. The lemma\ref{le:1} states that
$$
   \mathbf{A}_{d+1,k,d-j+1}\, =\, d!\cdot [
\lambda^j]\left(
(\lambda+1)^dB_{d+1}\left(k+\frac{1}{\lambda+1}\right)\right)/{\binom{d}{j}},\,\,
\lambda\geq 0.
$$
which implies
\begin{equation}\label{eq lem1}
p(\lambda)=\sum_{j=0}^d(_j^d)\mathbf{A}_{d+1,k,d-j+1}\lambda^j
=d!(\lambda+1)^dB_{d+1}\left(k+\frac{1}{\lambda+1}\right).
\end{equation}
Taking the $j$-th derivative of both sides of equation (\ref{eq
lem1}), it holds:
\begin{eqnarray*}
\mathbf{A}_{d+1,k,d-j+1}\,
\,&=&\,\sum_{i=0}^j(d-j)!B^{(i)}_{d+1}\left(k+\frac{1}{\lambda+1}\right)\left(({\lambda+1})^d\right)
^{(j-i)}\left(({\lambda+1}\right)^{-1})^
{(i)}|_{\lambda=0}\\
\,&=&\,d!\sum_{i=0}^j(-1)^i
\frac{i!(d-j)!}{(d-j+i)!}B^{(i)}_{d+1}(k+1)\\
\,&=&\,d!\sum_{i=0}^j(-1)^i
\frac{1}{\binom{d-j+i}{i}}B^{(i)}_{d+1}(k+1).\\
\end{eqnarray*}
Combining Theorem \ref{limtB} with the definition of Hermit
polynomials (\ref{H}), we obtain that
\begin{eqnarray*}
\mathbf{A}_{d+1,[x_d],d-j+1}\, \,&=&\,d!\sqrt{\frac{6}{\pi
(d+1)}}\sum_{i=0}^j(-1)^i
\frac{1}{\binom{d-j+i}{i}}\left(\frac{d+1}{12}\right)^{-\frac{i}{2}}D^i\exp\left({-\frac{x^2}{2}}\right)
+O\left(d^{\frac{-3}{2}}\right)\\
\, \,&=&\,d!\sqrt{\frac{6}{\pi
(d+1)}}\exp{\left(-\frac{x^2}{2}\right)}\sum_{i=0}^j
\frac{1}{\binom{d-j+i}{i}}\left(\frac{d+1}{12}\right)^{-\frac{i}{2}}H_i(x)+O\left(d^{\frac{-3}{2}}\right).\\
\end{eqnarray*}
\end{proof}
\begin{proof} [Proofs of Theorem \ref{th1} and Theorem\ref{corollary1}]
By using Lemma \ref{le:1}
$$
A_{d,k} = d!\cdot B_{d+1}(k),\leqno{\rm (i)}
$$
$$
 D(d,n,k)\,=\, d!\cdot n^d\cdot B_{d+1}\left(k+\frac{1}{n}\right),\leqno{\rm (ii)}
$$
and Corollary \ref{corollary2}
\begin{equation*}
\sqrt{\frac{d}{12}}B_{d}\left(\sqrt{\frac{d}{12}}x+\frac{d}{2}\right)=\frac{1}{\sqrt{2\pi}}
\exp{\left(-\frac{x^2}{2}\right)}+O\left(\frac{1}{d}\right),
\end{equation*}
 consequently, for $x_d=\sqrt{\frac{d+1}{12}}x+\frac{d+1}{2}$, we have
\begin{equation*}\label{eq:6}
\frac{1}{d!}A_{d,[x_d]}
=\sqrt{\frac{6}{\pi(d+1)}}\exp\left({-\frac{x^2}{2}}\right)+O\left(d^{-\frac{3}{2}}\right)
\end{equation*}
and
\begin{equation*}\label{eq:7}
\frac{1}{d!\cdot n^d}D(d,n,[x_d])=
\sqrt{\frac{6}{\pi(d+1)}}\exp\left({-\frac{\left(x+\frac{1}{n}\right)^2}{2}}\right)
+O\left(d^{-\frac{3}{2}}\right).
\end{equation*}
\end{proof}
\begin{remark} We can also prove Theorem \ref{th1} by using Theorem \ref{ThREN}.
 By the combinatorial interpretations of $A_{d,k}$ and
$\mathbf{A}_{d,k,j}$, we have
 $A_{d+1,k}=\mathbf{A}_{d+1,k,d+1}$. Combining with Theorem
\ref{ThREN}, one obtains
\begin{equation}
\frac{1}{d!}A_{d,[x_d]}=\sqrt{\frac{6}{\pi
(d+1)}}\exp\left({-\frac{x^2}{2}}\right)+O\left(d^{-\frac{3}{2}}\right).
\end{equation}
\end{remark}
\bigskip
{\small \noindent {\bf Acknowledgments.}  The authors are most
grateful to Professor Zhiqiang Xu, Wenchang Chu, Yang Wang, T. Y. Li and A. Ron for discussions and comments which improve the manuscript.

\end{document}